\documentclass{amsart}
\usepackage{amsmath}
  \usepackage{amssymb}
  \usepackage{paralist}
  \usepackage{graphics} %% add this and next lines if pictures should be in esp format
  \usepackage{epsfig} %For pictures: screened artwork should be set up with an 85 or 100 line screen
\usepackage{graphicx}  \usepackage{epstopdf}%This is to transfer .eps figure to .pdf figure; please compile your paper using PDFLeTex or PDFTeXify. 
 \usepackage[colorlinks=true]{hyperref}
\hypersetup{urlcolor=blue, citecolor=red}

\newtheorem{theorem}{Theorem}[section]

\newtheorem{lemma}[theorem]{Lemma}
\newtheorem{proposition}{Proposition}

\theoremstyle{definition}

\newtheorem{remark}{Remark}

\newcommand{\Rot}{{\mathrm{Rot}}}
\newcommand{\ITM}{{\mathrm{ITM}}}
\newcommand{\TITM}{{\mathrm{TITM}}}

\newcommand{\argmin}{\mathop{{\mathrm{arg\,min}}}\nolimits}
\newcommand{\argmax}{\mathop{{\mathrm{arg\,max}}}\nolimits}

\newcommand{\Tight}{{\mathcal{F}}}

\newcommand{\Leb}{\mathop{{\mathrm{Leb}}}\nolimits}

\newcommand{\bbR}{\mathbb{R}}

\newcommand{\bbN}{\mathbb{N}}

\newcommand{\dd}{\Delta}

\newcommand{\Om}{\Omega}
\newcommand{\Beta}{{B}}

\title[Interval Translation Maps]
      {Almost every Interval Translation Map of three intervals is finite type}

\author[Denis Volk]{}


 \email{dvolk@kth.se}

\thanks{The author was supported in part by grants RFBR 12-01-31241-mol\_a, President’s of Russia
MK-2790.2011.1, ``Young SISSA Scientists'', and NSF IIS-1018433 (PI Max
Welling, Co-PI Anton Gorodetski).}

\begin{document}
\maketitle

% Enter the first author's name and address:
\centerline{\scshape Denis Volk }
\medskip
{\footnotesize
% please put the address of the first author
 \centerline{Kungliga Tekniska H\"ogskolan (Royal Institute of Technology)}
   \centerline{Department of mathematics}
   \centerline{SE-100 44, Stockholm, Sweden}
} % Do not forget to end the {\footnotesize by the sign }

\bigskip

%% UNCOMMENT ME (AND OTHERS!!!)
%% The name of the associate editor will be entered by an editorial staff
%% "Communicated by the associate editor name" is not needed for special issue.
% \centerline{(Communicated by the associate editor name)}

%The abstract of your paper
\begin{abstract}
Interval translation maps (ITMs) are a non-invertible generalization of interval exchange transformations (IETs).
The dynamics of finite type ITMs is similar to IETs, while infinite type ITMs are known to exhibit new interesting effects. In this paper, we prove the finiteness conjecture for the ITMs of three intervals. Namely, the subset of ITMs of finite type contains an open, dense, and full Lebesgue measure subset of the space of ITMs of three intervals. For this, we show that any ITM of three intervals can be reduced either to a rotation or to a double rotation.
\end{abstract}

\section{Introduction and main result}  \label{s:intro}

\subsection{Interval translation maps}

Let $\Om \subset \bbR$ be a semi-interval split into $d$ disjoint semi-intervals, $\Om = \sqcup_{j=1}^d \Delta_j$, $\dd_j = [\beta_{j-1}, \beta_j)$. An \emph{interval translation} $T \colon \Om \to \Om$ is a map given by a translation on each of $\Delta_j$, $T|_{\Delta_j} \colon x \mapsto x+\gamma_j$, the vector $(\gamma_1, \dots, \gamma_d)$ is fixed. After normalization~$\Om := [0,1)$, the space~$\ITM(d)$ of $d$ intervals' translations is a convex polytope in~$\bbR^{2d-1}$. We endow it with the Euclidean metric and the Lebesgue measure. An example of an ITM is shown in Figure~\ref{f:itm}. We draw $\Om$ as a horizontal line, the splitting of $\Om$ into $\dd_j$ as arcs of different styles atop of the line, and the images of $\dd_j$ as the mirrored arcs below the line.

\begin{figure}[htp]
\begin{center}
\includegraphics[width=4in]{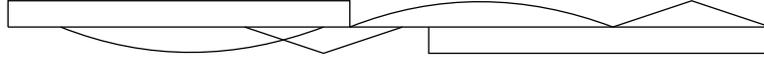}\\
\caption{An interval translation map of $3$ intervals.}
\label{f:itm}
\end{center}
\end{figure}

Interval translation maps (ITMs) were first introduced in 1995 by Boshernitzan and Kornfeld in~\cite{Boshernitzan1995}. They are a generalization of interval exchange transformations. Unlike IETs the ITMs are generally not invertible.
%These mathematical objects proved to be applicable to study of \textbf{(more blabla)} magnetization, discretized electronic devices~\cite{Teplinskiy2008}.
%As a rigid model for piecewise monotonic maps without periodic points, ITMs form an interesting class of zero entropy systems~\cite{Buzzi2004}.
%They also arise in the dynamics of polygonal billiards with semipermeable walls (spy mirrors)~\cite{Schmeling2000} similar to the way IETs arise in polygonal billiards.
%Double rotations appear in the field of electrical engineering as well. Most of the partial discharge models phenomena are based on
%simple three-capacitance equivalent circuit models of partial discharge~\cite{Whitehead1953} which can
%be reduced to a double rotation.
%ITMs are closely related to interval identification systems which arise from the study of conductivity
%theory of monocrystals in magnetic field (see~\cite{Maltsev2004}, \cite{Dynnikov2008} for details).

Define $\Om_0 = \Om$, $\Om_n = T \Om_{n-1}$ for $n\ge 1$, and let $X$ be the closure of $\cap_{n=1}^\infty \Om_n$. An interval translation map is called of \emph{finite type} if $\Om_{n+1} = \Om_n$ for some $n$, otherwise it is called of \emph{infinite type}. Denote the set of infinite type ITMs by $\mathcal{S}$.

Schmeling and Troubetzkoy~\cite{Schmeling2000} proved that for any $d$, an ITM is of finite type iff $X$ consists of a finite union of intervals. In this case, some power of the restriction~$T|_X$ is an interval exchange transformation. If $T$ is of infinite type and $T|_X$ is transitive, then~$X$ is a Cantor set. Additionally, Bruin and Troubetzkoy~\cite{Bruin2003} showed that if $X$ is transitive, then it is minimal i.e. every orbit is dense.

In their pioneering paper~\cite{Boshernitzan1995}, Boshernitzan and Kornfeld gave the first example of an ITM of infinite type. They also asked, \emph{``To what extent the ITMs of finite type are typical? In particular, is it true that `almost all' ITMs are of finite type?''} Then they remarked that \emph{the answer is affirmative for $d = 3$} but gave no proof of that.

Until now, the question was answered (positively) only for some specific families of ITMs: by Bruin and Troubetzkoy~\cite{Bruin2003} for a 2-parameter family in $\ITM(3)$, by Bruin~\cite{Bruin2007} for a $d$-parameter family in $\ITM(d+1)$, and by Suzuki, Ito, Aihara~\cite{Suzuki2005} and Bruin, Clack~\cite{Bruin2012} for so-called double rotations, see Subsection~\ref{ss:d-r-def}. Schmeling and Troubetzkoy~\cite{Schmeling2000} established a number of related topological results, in particular, that the interval translation mappings of infinite type form a $G_\delta$ subset of the set of all interval translation mappings, whereas the interval translation mappings of finite type contain an open subset.

%showed that if the dimension of the vector space~$\langle \beta_i, \gamma_i \rangle_\bbQ$ is less than 3, then the ITM is of finite type. This motivated the \emph{finiteness problem}: how big is the set of ITMs of infinite type?

But the original question stood open. In this paper, we finally justify the affirmative answer for general ITMs of three intervals:
\begin{theorem} \label{t:main}
  In the space~$\ITM(3)$, the set~$\mathcal{S} \cap \ITM(3)$ has zero Lebesgue measure.
\end{theorem}
\begin{remark}
  The numerical estimate by~\cite{Bruin2012} infers that its Hausdorff dimension~$H$ satisfies $4 \le H \le 4.88$.
\end{remark}
The same question about~$d > 3$ intervals remains open.

The entropy and word growth properties of piecewise translations, including their higher dimension generalizations, were studied in~\cite{Goetz1999}, \cite{Goetz2000}, \cite{Buzzi2001}, \cite{Goetz2001}. The question of their (unique) ergodicity was considered by \cite{Bruin2003}, \cite{Buzzi2004}, \cite{Bruin2012}.

%The Hausdorff dimension measure on $X$ is $T$-invariant. However it is unknown if it is finite in general.
%
%\textbf{[INFO]} In~\cite{Bruin2003} it was shown that there are infinite type ITMs that are not uniquely ergodic, although the number of ergodic measures is bounded by $(d + 1)/2$ (for his Bruin-Troubetzkoy-like family), see~\cite{Buzzi2004}. Hence for the ITMs in this paper, two is the
%maximal number of ergodic measures. Although Rauzy induction, and little of the other advanced machinery of IETs are available for ITMs, non-unique ergodicity turns out to be extremely rare: \cite{Bruin2011} prove that for every (Suzuki induction)-invariant measure $\mu$ such that $\mu(Abyss) = 0$, $\mu$-a.e. parameter corresponds to a uniquely ergodic ITM. \textbf{[-INFO]}

We always assume $\gamma_i \ne 0$ and $\beta_i - \beta_{i-1} \ne 0$ for all $1 \le i \le d$.

\subsection{Tight ITMs}

%In this paper, the word interval will always refer to an interval which is closed on the left, open on the right.

Let~$T$ be any interval translation map. We say that an interval~$\Delta \subset \Om$ is \emph{T-regular} (or simply \emph{regular}) if there exists $N \in \bbN$ such that for any $x \in [0,1)$ there exists $1 \le n < N$ such that $T^n x \in \Delta$. In particular, every point of~$\Delta$ returns to $\Delta$ after a uniformly bounded number of iterates of $T$, so the \emph{induced (i.e. first-return)} map of $T$ is well defined on $\Delta$. We denote the induced map by~$T_\Delta$. We say that an interval $\Delta \subset \Om$ is a \emph{trap} if it is regular and $T \dd \subset \dd$. In this case, we have $T_\dd = T|_\dd$.

The following Lemma shows that the properties to be finite or infinite type are preserved by inductions.

\begin{lemma} \label{l:iff-finite}
Assume $T|_X$ is transitive. Then,
\begin{enumerate}
  \item\label{en:iff1} if there exists a regular~$\dd$ such that $T_\dd$ has finite type, then $T$ has finite type;
  \item\label{en:iff2} if~$T$ is finite type, then for any regular~$\dd$ the map $T_\dd$ has finite type.
\end{enumerate}
\end{lemma}

\begin{proof}
  First we prove \ref{en:iff1}. Assume~$T$ has infinite type. Then $X$ is a $T$-invariant Cantor set with minimal dynamics (see~\cite{Schmeling2000}, \cite{Bruin2003}). Because $\dd$ is regular, $Y = X \cap \dd$ is a nonempty Cantor set. It is also $T_\dd$-invariant, and thus $Y \subset X(T_\dd)$. Because $T_\dd$ has finite type, $X(T_\dd)$ is a finite union of intervals with an IET on them. This implies $Y \ne X(T_\dd)$. Because $X$ is transitive, it is minimal for $T$, and thus $X(T_\dd)$ is minimal for $T_\dd$. But we have just shown that it contains another invariant closed set. This contradiction proves the lemma.

  Recall that $X$ is a finite union of intervals with an IET on them. So, because $X$ is transitive, it is also minimal.

  Now we prove \ref{en:iff2}. Assume~$T_\dd$ has infinite type. Then $X(T_\dd)$ is a $T_\dd$-invariant transitive set. Denote by $Y$ the union of its $N$ iterates by~$T$:
  $$
    Y = \bigcup_{n=1}^N T^n X(T_\dd).
  $$
  Now $Y$ is a $T$-invariant closed set and $Y \varsubsetneq X$. This contradicts the minimality of $X$ and thus proves the lemma.
\end{proof}

%\begin{lemma}
%  For any~$T \in \ITM(d)$ and $\Delta = [T\Om)$, the map~$T_\Delta$ is an interval translation map of no more than~$d$ intervals.
%\end{lemma}
%
%\begin{proof}
%  Obvious. \textbf{(?)} \textbf{But why I need this?}
%\end{proof}
%
%Now the results of this paper can be summarized as follows:
%\begin{theorem} \label{t:main}
%  There exists an open and dense set $U \subset \ITM(3)$ such that
%  \begin{itemize}
%    \item for any $T \in U$ there exists a regular $\Delta$ such that $T_\Delta$ is a double rotation;
%    \item the map $T \mapsto T_\Delta$ is \textbf{good (piecewise smooth, full rank..).}
%  \end{itemize}
%\end{theorem}
%The proof of Theorem~\ref{t:main} is the combination of Lemma~\ref{l:fitting} and Theorem~\ref{t:titm-d-r}.

For any~$M \subset [0,1]$, we denote $[M) := [\inf M, \sup M)$. We say that an interval translation map $T\colon \Om \to \Om$ is \emph{tight} if $[T\Om) = [\Om)$. We denote the space of tight interval translation maps of $d$ intervals by~$\TITM(d)$. Recall that $\ITM(d)$ is a filled convex polytope in $\bbR^{2d-1}$, and observe that $\TITM(d)$ is a connected finite union of filled convex polytopes in $\bbR^{2d-3}$.

%\begin{lemma}    \label{l:fitting}
  It is easy to see that for any~$T \in \ITM(d)$, there exists a trap~$\Delta$ such that the map~$T_\Delta$ is a tight interval translation map of~$r$ intervals, $r \le d$.
%\end{lemma}
Namely, take $X_0' := \Omega$, and let $X_k' := [T X_{k-1}')$, $\Delta := \bigcap_{k=0}^\infty X_k'$. Then $\Delta \supset X$, $\Delta$ is a trap, and $T_\Delta$ is tight.

However, to prove Theorem~\ref{t:main}, we will need more explicit way to produce such a $\Delta$.

%\begin{proof}
  For $d = 2$, we have $\Delta = [\beta_1 + \gamma_2, \beta_1 + \gamma_1)$, see Figure~\ref{f:rotation} and keep in mind that $\gamma_1 > 0$, $\gamma_2 < 0$.
  \begin{figure}[htp]
    \begin{center}
    \includegraphics[width=4in]{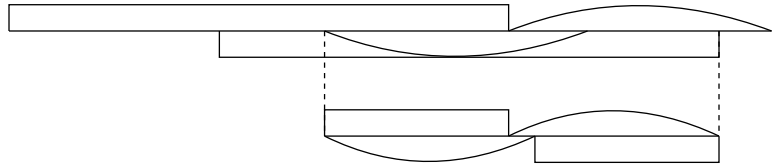}\\
    \caption{ITM of two intervals: rotation.}   \label{f:rotation}
    \end{center}
  \end{figure}
  Indeed, for any $x < \beta_1$ we have $Tx = x+\gamma_1$. Thus there exists $1 \le n_1 \le \left[\frac{\beta_1}{\gamma_1}\right]+1$ such that $T^{n_1} x \in [\beta_1, \beta_1 + \gamma_1)$. Similarly, for any $x \ge \beta_1$ there exists $1 \le n_2 \le \left[\frac{1-\beta_1}{|\gamma_2|}\right]+1$ such that $T^{n_2} x \in [\beta_1 + \gamma_2, \beta_1)$. Thus $\dd$ is regular.
  Now note that $T|_\dd$ is just a rotation $x \mapsto x+\gamma_1 \mod \dd$, so $T\dd = \dd$. Thus $\dd$ is a trap and $T_\dd$ is tight.

  %Assume we have the statement for any~$k < d$. Let us prove it for $d$ then.
  To obtain a similar formula for $d \ge 3$,
  consider two cases:
  \begin{align}
    \label{case:reduc} \dd_1 \cap T\Om = \emptyset \text{ or } \dd_d \cap T\Om = \emptyset;\\
    \label{case:irreduc} \dd_1 \cap T\Om \ne \emptyset \text{ and } \dd_d \cap T\Om \ne \emptyset.
  \end{align}
  In the first case, we can completely remove $\dd_1$ or $\dd_d$ and thus reduce the problem to the study of $\ITM(d-1)$.
  %$k = d-1$.
  For $\ITM(3)$, this case can be ignored because then the map is necessary a rotation, and thus finite type.

  In the second case, let $I_-$ be the set $\{ i \,|\, \gamma_i < 0 \}$ and $I_+$ be the set $\{ i \,|\, \gamma_i > 0 \}$. Because~$\gamma_1 > 0$ and $\gamma_d < 0$, the both sets are nonempty. Take the interval
$$
\Delta = [\delta_0, \delta_1) = [ \min_{i \in I_-} (\beta_{i-1} + \gamma_i), \max_{i \in I_+} (\beta_i + \gamma_i)),
$$
\begin{figure}[htp]
    \begin{center}
    \includegraphics[width=4in]{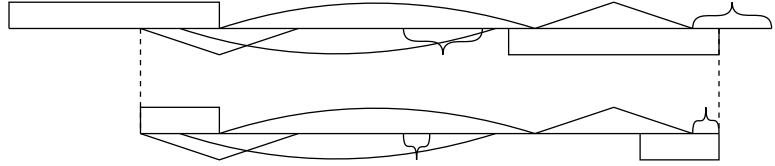}\\
    \caption{Fitting operator.}   \label{f:fitting}
    \end{center}
  \end{figure}
  see Figure~\ref{f:fitting}.
  In the same way as for $d=2$ one can show that for any $x < \delta_0$ or $x > \delta_1$ there exists a bounded $n$ such that $T^n x \in \dd$. Now, for any $x \in \dd$, $x$ either moves to the left or to the right. Assume it moves to the left. Then $\min_{i \in I_-} (\beta_{i-1} + \gamma_i) \le T x < x$ which implies $T x \in \dd$. Similarly, for $x$ moving right we have $x < T x \le \max_{i \in I_+} (\beta_i + \gamma_i)$. Thus $T x \in \dd$, and $\dd$ is a trap.

  Because we are in the second case, for every $2 \le i \le d$, the left end~$\beta_{i-1}$ of~$\dd_i$ belongs to $\dd$. Moreover, for $i = \argmin_{i \in I_-} (\beta_{i-1} + \gamma_i)$ the left end of $\Delta_i$ maps exactly to $\delta_0$. (Here and below, $\argmin$ stands for ``argument of the minimum'', i.e., $\argmin_{x \in S} f(x) := \{ x \in S \,|\, \forall y \in S \, f(y) \ge f(x) \}$, and the same with $\argmax$.) Thus $T_\dd$ is tight from the left. Similarly, we show $T_\dd$ is tight from the right and thus tight in general.
%\end{proof}

Rescale the map~$T_\Delta$ so that $\dd$ becomes $[0,1)$. We say that the result, $\tilde T_\dd$, is the \emph{fitting} of $T$ and denote the \emph{fitting operator} $T \mapsto \tilde T_\Delta$ by $\Tight_d \colon \ITM(d) \to \TITM(d)$. When the value of $d$ is clear, we will write $\Tight$ instead of $\Tight_d$ for brevity.

\begin{proposition} \label{p:rank}
  In the case~\eqref{case:irreduc}, the fitting operator is a piecewise rational map of the maximal rank $2d-3$.
\end{proposition}

\begin{proof}
  We partition the space~$\ITM(d)$ into the union of the cells~$C_{jk}$, $j \ne k$:
  $$
    C_{jk} = \{ T \in \ITM(d) \,|\, j = \argmin_{i \in I_-} (\beta_{i-1} + \gamma_i), k = \argmax_{i \in I_+} (\beta_i + \gamma_i) \}.
  $$
  The fitting operator is the composition of truncation and rescaling:
  $$
    \mathcal{F} = \mathcal{R} \circ \mathcal{T}.
  $$
  The rescaling part is a rational map~$\mathcal{R}\colon \bbR^{2d-1} \to \bbR^{2d-3}$ of rank $2d-3$. To understand the truncation part~$\mathcal{T}$, we introduce new coordinates~$\Beta_{i-1} = \beta_{i-1} + \gamma_{i}$, $i = 1,\dots,d$, which are the images of the left ends of~$\Delta_i$. In particular, $\Beta_0 = \gamma_0$. At every cell~$C_{jk}$, the truncation is a linear map~$\mathcal{T} \colon \bbR^{2d-1} \to \bbR^{2d-1}$. Let us show it is invertible. In the coordinates~$(\beta_i, \Beta_i)$, the truncation has the form
  \begin{equation}  \label{e:trunc}
    \begin{array}{c}
      \left[\begin{array}{c}
    0, \beta_1, \beta_2, \dots, \beta_{d-2}, \beta_{d-1}, 1 \\
    \Beta_0, \Beta_1, \dots, \Beta_{j-1}, \dots, \Beta_k, \dots, \Beta_{d-2}, \Beta_{d-1} \\
      \end{array}\right] \\
    \downarrow\,\mathcal{T} \\
    \left[\begin{array}{c}
    \Beta_{j-1}, \beta_1, \beta_2, \dots, \beta_{d-2}, \beta_{d-1}, \Beta_k \\
    \Beta_{j-1}+\Beta_0, \Beta_1, \dots, \Beta_{j-1}, \dots, \Beta_k, \dots, \Beta_{d-2}, \Beta_k + \Beta_{d-1} - \beta_{d-1}.
    \end{array}\right] \\
        \end{array}
  \end{equation}
  Note that most of the coordinates are mapped identically. The non-identical part of~$\eqref{e:trunc}$ is
  \begin{equation}  \label{e:trunc-reduced}
    \begin{array}{c}
    \left[\begin{array}{c}
    0, \beta_{d-1}, 1 \\
    \Beta_0, \Beta_{j-1}, \Beta_k, \Beta_{d-1} \\
    \end{array}\right] \\
    \downarrow\,\mathcal{T} \\
    \left[\begin{array}{c}
    \Beta_{j-1}, \beta_{d-1}, \Beta_k \\
    \Beta_{j-1}+\Beta_0, \Beta_{j-1}, \Beta_k, \Beta_k + \Beta_{d-1} - \beta_{d-1}
    \end{array}\right] \\
        \end{array}
  \end{equation}
  which is clearly invertible. So, at every cell~$C_{jk}$ the fitting~$\mathcal{F}$ is a rational map of rank~$2d-3$.
\end{proof}

All the previous considerations are valid for any number of intervals. In what follows, we have the proofs only for the case of $d=3$.

\subsection{Double rotations}\label{ss:d-r-def}

Following Suzuki, Ito, Aihara~\cite{Suzuki2005}, we introduce the family of double rotations.
A \emph{double rotation} is a map~$f_{(a, b, c)} \colon [0,1) \to [0,1)$ defined by
$$
f_{(a, b, c)} (x) = \begin{cases}
  \{ x + a \}, & \text{if $x \in [0, c)$,} \\
  \{ x + b \}, & \text{if $x \in [c, 1)$.} \\
\end{cases}
$$
In the circle representation $[0,1) \equiv S^1$, a double rotation is a map $S^1 \to S^1$ defined by independent rotations of two complementary arcs of~$S^1$. Clearly, any double rotation is an ITM of two to four intervals. Denote the space of parameters~$(a, b, c)$ of double rotations by $\Rot(2) = [0,1) \times [0,1) \times [0,1)$. 
%In the same way we define the family~$\Rot(n)$ of \emph{$n$-rotations}.

For the double rotations, the finiteness problem is solved by Bruin and Clack in \cite{Bruin2012} using the renormalization operator of Suzuki, Ito, Aihara~\cite{Suzuki2005}. Namely, they proved the following
\begin{theorem}[Bruin, Clack, 2012] \label{t:bruin-clack}
  In the space~$\Rot(2)$, the set~$\mathcal{R} = \mathcal{S} \cap \Rot(2)$ of ITMs of infinite type has zero Lebesgue measure.
  %Its Hausdorff dimension~$H_\mathcal{R}$ satisfies $2 \le H_\mathcal{R} \le 2.88$.
\end{theorem}
It also follows from their construction that $\mathcal{R}$ is contained in a closed nowhere dense subset of~$\Rot(2)$.

In the present paper we show that the finiteness problem for~$\ITM(3)$ reduces to the one for the double rotations.

\begin{theorem} \label{t:titm-d-r}
  The space $\TITM(3)$ splits into countably many open sets~$A$, $A'$, $B$, $B_i$, $C_i$, $i \in \bbN$, such that their union~$U$ is dense in $\TITM(3)$, and the complement~$K = \TITM(3) \setminus U$ which is a union of countably many hyperplanes. Moreover,
\begin{itemize}
  \item any $T \in A \cup A'$ is a double rotation,
  \item any $T \in B$ is reduced to a double rotation via a single Type 1 induction,
  \item for any $i \in \bbN$, any $T \in B_i$ is reduced to a double rotation via a single Type 2 induction.
  \item for any $i \in \bbN$, any $T \in C_i$ is reduced to a single rotation via a single Type 2 induction.
\end{itemize}
  On every piece~$A, A', B, B_i, C_i$, these inductions are invertible rational maps.
\end{theorem}

%\begin{conjecture}
%  The space $\TITM(d)$ splits into countably many open sets~$A_i$ such that their union~$U$ is dense in $\TITM(d)$, and the complement~$K_d = \TITM(d) \setminus U$ which is a union of countably many hyperplanes. Any~$T \in U$ is reduced to an $(d-1)$-rotation via some induction.
%\end{conjecture}

Because the inductions are local diffeomorphisms, the preimages of zero measure sets are zero measure sets. The Hausdorff dimension is also preserved. This allows to transfer the results of~\cite{Suzuki2005} and~\cite{Bruin2012} to $\TITM(3)$.

\section{Proof of Theorem~\ref{t:titm-d-r}}

Let~$T \in \TITM(3)$ with $\Om = [0,1)$. Without loss of generality, we can assume $\gamma_1 > 0$ and $\gamma_3 < 0$. Because $T$ is tight, some interval (not $\dd_1$) must go to the leftmost position, and some interval (not $\dd_3$) must go to the rightmost position. Obviously, there are $3$ cases:
$$
\begin{array}{lccc}
  & A & A' & B\&C \\
  Leftmost & \dd_2 & \dd_3 & \dd_3 \\
  Rightmost & \dd_1 & \dd_2 & \dd_1
\end{array}
$$
The cases $A$ and $A'$ are mirror images of each other, so we consider only case $A$ of these two.

\subsection{Case A. Double rotation in disguise}  \label{ss:d-r}

In this case, $\Delta_2$ goes to the leftmost position and $\Delta_1$ goes to the rightmost position, see Figure~\ref{f:d-r-disguise}.
%
%Because $T$ is tight, some interval must also be mapped to the rightmost position. It cannot be $\Delta_2$. Neither it can be $\Delta_3$, because $\gamma_3 < 0$. So it is $\Delta_1$.
%
%Now $\Delta_1$ goes to the rightmost position and $\Delta_2$ goes to the leftmost one, see Figure~\ref{f:d-r-disguise}.
Then $T$ is a double rotation with $c = \beta_2$ (i.e. the first arc is $\dd_1 \cup \dd_2$ and the second one is~$\dd_3$) and $a = -|\Delta_1|$, $b = \gamma_3$.

\begin{figure}[htp]
\begin{center}
\includegraphics[width=4in]{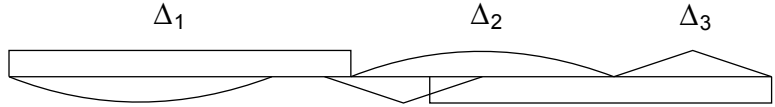}\\
\caption{Double rotation in disguise.}   \label{f:d-r-disguise}
\end{center}
\end{figure}

\subsection{Cases B and C. Induction}    \label{ss:induc}

In this case, $\Delta_1$ goes to the rightmost position and $\Delta_3$ goes to the leftmost position. Because of the symmetry, we can assume without loss of generality that $|\Delta_1| \ge |\Delta_3|$. Consider the two sub-cases: $\gamma_2 < 0$ and $\gamma_2 > 0$.

\subsubsection*{\underline{Sub-Case $\gamma_2 < 0$ [piece B]}}

\begin{proposition}
  In this case, $\Delta = \Delta_1 \cup \Delta_2$ is regular with the return time $\le 2$. $T_\Delta$ is a tight ITM of three intervals which is a double rotation.
\end{proposition}

\begin{figure}[htp]
\begin{center}
\includegraphics[width=4in]{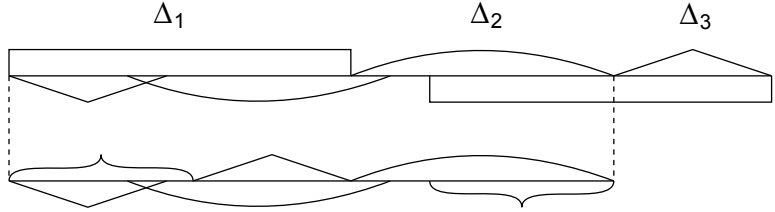}\\
\caption{Induction to $\Delta_1 \cup \Delta_2$.}   \label{f:i-left}
\end{center}
\end{figure}

\begin{proof}
  $\gamma_2 < 0$ implies $T\Delta_2 \subset \Delta$, so $\dd_2$ returns to $\dd$ in one piece after a single iteration of $T$. On the other hand, $\Delta_1$ is split by the first return map into two pieces, $\dd_1'$ and $\dd_1''$, separated by the point $\beta_2 - \gamma_1$, see Figure~\ref{f:i-left}. For the first piece,
  $$
    T \dd_1' = T\dd_1 \cap \dd \subset \dd.
  $$
  For the second piece, we have
  $$
    T \dd_1'' = T\dd_1 \setminus \dd = \dd_3, \quad \text{and} \quad T^2 \dd_1'' = T\dd_3 \subset \dd_1 \subset \dd.
  $$
  Here we used $|\dd_1| \ge |\dd_3|$. We have just shown that $\dd$ is regular. $T_\dd$ is a tight ITM of three intervals, $\dd_1', \dd_1'',\dd_2$. Note that $T_\dd$ sends $\dd_1'$ to the rightmost position and $\dd_1''$ to the leftmost one. Thus we have Case $A$ for $T_\dd$, and $T_\dd$ is a double rotation.

  Note that the induction operator $T \mapsto T_\dd$ is an invertible rational map on~$B$.
\end{proof}

\begin{remark}
  The sub-case $\gamma_2 > 0$ is a generalized Bruin-Troubetzkoy family~\cite{Bruin2003} with one extra degree of freedom, $\gamma_2$.
\end{remark}

\subsubsection*{\underline{Sub-Case $\gamma_2 > 0$ [pieces $B_i$, $C_i$]}}

\begin{proposition}
  In this case, $\Delta = \Delta_2 \cup \Delta_3$ is regular, and $T_\Delta$ is a tight ITM of three intervals which is a double rotation.
\end{proposition}

\begin{proof}
  Similarly, $\gamma_2 > 0$ implies $T\Delta_2 \subset \Delta$, so $\dd_2$ returns to $\dd$ in one piece after a single iteration of $T$.
  Let us now observe the return of $\dd_3$.

  We know $T \dd_3$ is at the leftmost position, which implies $T\dd_3 \subset \dd_1$. Thus $T^2\dd_3 = T\dd_3 + \gamma_1$. Moreover, for any $n \in \bbN$ such that $T^n \dd_3 \subset \dd_1$, we have $T^n\dd_3 = T\dd_3 + (n-1)\gamma_1$. Let $n$ be the maximal $n$ such that $T^2\dd_3, \dots, T^n \dd_3 \subset \dd_1$. There are two possibilities:
  \begin{enumerate}
    \item\label{case:b} $T^{n+1} \dd_3 \subset \dd$, or
    \item\label{case:c} $T^{n+1}\dd_3 \cap \dd \ne \emptyset$, $T^{n+1}\dd_3 \not\subset \dd$.
  \end{enumerate}

  Possibility~\ref{case:b} corresponds to~$B_i$, $i = n$. In this case, $\dd$ is obviously regular, and $T_\dd \in \ITM(2)$, see Figure~\ref{f:i-right-1}. By~\cite{Boshernitzan1995}, $T_\dd$ reduces to a (single) rotation.
  
  \begin{figure}[htp]
    \begin{center}
    \includegraphics[width=4in]{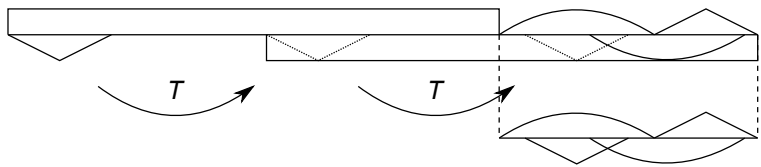}\\
    \caption{Induction to $\Delta_2 \cup \Delta_3$: possibility~\ref{case:b}.}   \label{f:i-right-1}
    \end{center}
  \end{figure}

  Possibility~\ref{case:c} corresponds to~$C_i$, $i = n$. In this case, $\dd_3$ is split into two pieces $\dd_3'$ and $\dd_3''$. $\dd_3''$ returns to $\dd$ after $n+1$ iterations of $T$ and goes to the leftmost position, see Figure~\ref{f:i-right-2}. $\dd_3'$ returns to $\dd$ after $n+2$ iterations and goes to the rightmost position. Thus we have Case $A'$ for $T_\dd$, and $T_\dd$ is a double rotation.

  \begin{figure}[htp]
    \begin{center}
    \includegraphics[width=4in]{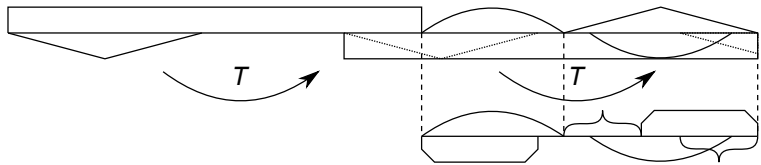}\\
    \caption{Induction to $\Delta_2 \cup \Delta_3$: possibility~\ref{case:c}.}   \label{f:i-right-2}
    \end{center}
  \end{figure}

  Note that the induction operator $T \mapsto T_\dd$ is an invertible rational map on each~$B_i$ or $C_i$.
\end{proof}

\section{Proof of Theorem~\ref{t:main}}

Note that (apart from a few trivial exceptions) for any $T \in \ITM(3)$ the limit set $X$ is either entirely periodic or transitive. All the infinite type ITMs belong to the latter case. Also for $\ITM(3)$ we can ignore the case~\eqref{case:reduc}, because in this case the ITM is just a rotation, and cannot be infinite type. Recall that by Lemma~\ref{l:iff-finite}, the properties to be finite or infinite type are preserved by inductions.

Now we invoke Theorem~\ref{t:bruin-clack} by Bruin and Clack~\cite{Bruin2012}. Together with our Theorem~\ref{t:titm-d-r} it implies that the set $\mathcal{S} \cap \TITM(3)$ is a countable union of sets of zero Lebesgue measure. \cite{Bruin2012} also numerically estimated the Hausdorff dimension of the infinite type parameters to be between $2$ and $2.88$. Thus,
$$
\Leb(\mathcal{S} \cap \TITM(3)) = 0, \quad 2 \le \dim(\mathcal{S} \cap \TITM(3)) \le 2.88.
$$
By Proposition~\ref{p:rank}, the set $\mathcal{S} \cap \ITM(3) = \mathcal{F}^{-1} (\mathcal{S} \cap \TITM(3))$ has zero Lebesgue measure and Hausdorff dimension between $4$ and $4.88$. The main theorem is proven.

\section*{Acknowledgements}

The author is grateful to Professor A.~Gorodetski for many fruitful discussions, to University of California, Irvine, for hospitality during the initial work on this paper, and to anonymous referees whose comments helped to improve the paper. The graphical representations of ITMs are inspired by the figures of interval identifications by A.~Skripchenko.
% in~\cite{Skripchenko2012}.


\begin{thebibliography}{99}


\bibitem{Boshernitzan1995}
Michael Boshernitzan and Isaac Kornfeld,
\newblock {\em Interval translation mappings},
\newblock Ergodic Theory Dynam. Systems, \textbf{15(5)} (1995), 821--832.

\bibitem{Bruin2007}
Henk Bruin,
\newblock {\em Renormalization in a class of interval translation maps of {$d$}
  branches},
\newblock Dyn. Syst., \textbf{22(1)} (2007), 11--24.

\bibitem{Bruin2012}
Henk Bruin and Gregory Clack,
\newblock {\em Inducing and unique ergodicity of double rotations},
\newblock Discrete Contin. Dyn. Syst., \textbf{32(12)} (2012), 4133--4147.

\bibitem{Bruin2003}
Henk Bruin and Serge Troubetzkoy,
\newblock {\em The {G}auss map on a class of interval translation mappings},
\newblock Israel J. Math., \textbf{137} (2003), 125--148.

\bibitem{Buzzi2001}
J{\'e}r{\^o}me Buzzi,
\newblock {\em Piecewise isometries have zero topological entropy},
\newblock Ergodic Theory Dynam. Systems, \textbf{21(5)} (2001), 1371--1377.

\bibitem{Buzzi2004}
J{\'e}r{\^o}me Buzzi and Pascal Hubert,
\newblock {\em Piecewise monotone maps without periodic points: rigidity, measures
  and complexity},
\newblock Ergodic Theory Dynam. Systems, \textbf{24(2)} (2004), 383--405.

\bibitem{Goetz1999}
Arek Goetz,
\newblock {\em Sofic subshifts and piecewise isometric systems},
\newblock Ergodic Theory Dynam. Systems, \textbf{19(6)} (1999), 1485--1501.

\bibitem{Goetz2000}
Arek Goetz,
\newblock {\em Dynamics of piecewise isometries},
\newblock Illinois J. Math., \textbf{44(3)} (2000), 465--478.

\bibitem{Goetz2001}
Arek Goetz,
\newblock {\em Stability of piecewise rotations and affine maps},
\newblock Nonlinearity, \textbf{14(2)} (2001), 205--219.

\bibitem{Schmeling2000}
{J\"org} Schmeling and Serge Troubetzkoy,
\newblock {\em Interval translation mappings},
\newblock In Dynamical systems ({L}uminy-{M}arseille, 1998),
  291--302. World Sci. Publ., River Edge, NJ, 2000.

%\bibitem{Skripchenko2012}
%Alexandra Skripchenko,
%\newblock {\em Symmetric interval identification systems of order three},
%\newblock Discrete Contin. Dyn. Syst., \textbf{32(2)} (2012), 643--656.

\bibitem{Suzuki2005}
Hideyuki Suzuki, Shunji Ito, and Kazuyuki Aihara,
\newblock {\em Double rotations},
\newblock Discrete Contin. Dyn. Syst., \textbf{13(2)} (2005), 515--532.

\end{thebibliography}
\end{document}